\documentclass[12pt]{amsart}
\usepackage{amsfonts, amssymb, amsmath, amsthm}
\usepackage{enumerate}
\usepackage{color}
\usepackage[normalem]{ulem}
\newcommand{\Z}{\mathbb{Z}}
\newcommand{\N}{\mathbb{N}}
\newcommand{\R}{\mathbb{R}}
\newcommand{\rr}{\mathcal{R}}

\newcommand{\calR}{\mathcal{R}}
\newcommand{\ear}{E(A,\calR)}

\newcommand{\Rcal}{\mathcal{R}}

\newcommand{\As}{\mathcal{A^{\sharp}}}

\newtheorem{theorem}{Theorem}[section]
\newtheorem{proposition}[theorem]{Proposition}

\newtheorem{definition}[theorem]{Definition}

\newtheorem{spec}[theorem]{Spectral Conjecture}
\newtheorem{wgenspec}[theorem]{Generalized Spectral Conjecture
(weak form, 1991) }
\newtheorem{sgenspec}[theorem]{Generalized Spectral Conjecture
(strong form, 1993) }
\newtheorem{remark}[theorem]{Remark}

\newtheorem{lemma}[theorem]{Lemma}

\newtheorem{example}[theorem]{Example}

                            {\end{enumerate}}
\input xy
\xyoption{all}
\numberwithin{equation}{section}

\begin{document}

%This manuscript has been accepted for publication in Linear Algebra
%and Its Applications. The manuscript will undergo copyediting,
%typesetting, and review of the resulting proof before it is published
%in its final form. Please note that during the production process
%errors may be discovered which could affect the content, and all
%disclaimers that apply to the journal apply to this manuscript.
%\newpage

\keywords{nonnegative matrix; spectra; shift equivalence; spectral conjecture}
\subjclass[2010]{Primary 15B48; Secondary 37B10}

\title[SSE and GSC]{Strong shift equivalence and the Generalized Spectral Conjecture
for nonnegative matrices}
\author{Mike Boyle and Scott Schmieding}
%\authorrunning{Boyle and Schmieding}
%\titlerunning{Strong shift equivalence and the generalized spectral conjecture}

 \dedicatory{Dedicated to Hans Schneider, in  memoriam} 

 %%%THIS FILE IS SPEC10a WITHOUT THE HANS SCHNEIDER PRIZE FOOTNOTE
 %%%AND DEDICATION. 
 
\maketitle

%\mbred{VERSION SPEC10 SS 17 MAY 2014, plus dedication}
%%%% This is actually 10a, which is 10 with some color stripped. 
\begin{abstract}
Given matrices $A$ and $B$ shift equivalent
over a dense subring $\rr$  of $\R$, with
$A$  primitive, we show that $B$
is strong shift equivalent over $\rr$
to a primitive matrix. This result shows that the weak form
of  the Generalized Spectral Conjecture for primitive matrices
implies the strong form. The foundation of this work is the
recent  result that
for any ring $\rr$,
the group $\textnormal{NK}_1(\rr)$
of algebraic K-theory classifies
the refinement of shift equivalence
by strong shift equivalence for matrices over $\rr$.
\end{abstract}

\tableofcontents

\section{Introduction}

%The openness 
%of the award committee 
%to two  outsiders using very different methods 
%was striking,  
The purpose of this paper
\footnote{
This paper is an outgrowth of the paper \cite{BH91}, for which its 
authors were awarded (along with Robert Thompson) 
the second Hans Schneider Prize.} 
 is to prove the following theorem
and explain its context.

\begin{theorem}\label{twoconj}
Suppose $\rr$ is a dense subring of $\R$,
$A$ is a primitive matrix over $\rr$ and
$B$ is a matrix over $\rr$ which is shift equivalent
over $\rr$ to $A$.

Then $B$ is strong shift equivalent over  $\rr$ to a primitive matrix.

\end{theorem}
We begin with the context. By ring, we mean a ring with 1;
by a semiring, we mean a semiring containing $\{0,1\}$.
A primitive matrix is a square matrix
which is nonnegative (meaning
entrywise nonnegative) such that for some $k>0$ its $k$th power
is a positive matrix.
Definitions and more background for shift equivalence (SE) and
strong shift equivalence (SSE) are given in Section
\ref{sesect}.

We recall the Spectral Conjecture
for primitive matrices from \cite{BH91}.
In the statement,
$\Delta = (d_1, \dots ,d_k)$ is a
$k$-tuple of nonzero complex numbers.
$\Delta$ is the
{\it nonzero spectrum} of a matrix $A$ if $A$ has
characteristic polynomial of the form
$\chi_A(t)=t^m\prod_{1\leq i \leq k}(t-d_i)$.
$\Delta $ has a {\it Perron value} if there exists $i$
such that $d_i>|d_j|$ when $j\neq i$.
The
{\it trace} of $\Delta$ is
$\textnormal{tr}(\Delta ) =
d_1 + \cdots + d_k$.
$\Delta^n$ denotes $((d_1)^n, \dots ,(d_k)^n)$,
the tuple of $n$th powers;
and the {\it $n$th net trace} of $\Delta$ is
\[
\textnormal{tr}_n(\Delta ) = \sum_{d|n}\mu (n/d) \textnormal{tr}(\Delta ^d)
\]
in which $\mu$ is the M\"{o}bius function
($\mu (1) =1$; $\mu (n)= (-1)^r$ if $n$ is the product of $r$
distinct primes; $\mu (n) = 0$ if $n$ is divisible by the square
of a prime).

\begin{spec} \cite{BH91}
Let $\rr$ be a subring of $\R$. Then $\Delta$ is the nonzero spectrum
of some primitive matrix over $\rr$ if and only if the following
conditions hold:
\begin{enumerate}
\item
$\Delta$ has a Perron value.
\item The coefficients of the polynomial
$\prod_{i=1}^k(t-d_i)$ lie in $\rr$.
\item
If $\rr = \Z$, then for all positive integers $n$,
$\textnormal{tr}_n(\Delta )\geq 0 $; \\
if $\rr \neq \Z$, then for all positive integers $n$ and $k$,
\\ (i) $\textnormal{tr}(\Delta^n) \geq 0 $ and
 (ii) $\textnormal{tr}(\Delta^n) > 0 $ implies
$\textnormal{tr}(\Delta^{nk}) > 0 $.
\end{enumerate}
\end{spec}
It is not difficult to check that the nonzero spectrum
of a primitive matrix satisfies the three conditions \cite{BH91}.
(We remark, following \cite{JohnsonLaffeyLoewy1996}
it is known that the nonzero spectra
of symmetric primitive matrices
cannot possibly have such a
simple characterization.)

To understand the possible spectra of  nonnegative
matrices is a classical problem of linear algebra
(for early background see e.g. \cite{BH91}) on which
interesting progress continues (see e.g.
\cite{EllardSmigoc2013,
Laffey2012BHTheorem,
LLS2009, LLS13} and their references).     Understanding
the nonzero spectra of primitive matrices is a variant of
this problem and also an approach to it: to know the minimal size
of a primitive matrix with a prescribed nonzero spectrum is to solve the
classical problem (for details, see \cite{BH91}); and it is in
the primitive case that the Perron-Frobenius constraints
manifest most simply.

Finally, as the  spectra of matrices over various subrings
of $\R$ appear in applications,
in which the nonzero part of the spectrum is sometimes
the relevant part \cite{Boyle91matrices, BH91},  it is natural to consider
the nonzero spectra of matrices over arbitrary subrings of
 $\R$.

The Spectral Conjecture has been proved in enough cases that
it seems almost certain to be  true in general.
For example, it is true under any of the following conditions:
\begin{itemize}
\item
The Perron value of $\Lambda$ is in $\rr$ (this always
holds when $\rr=\R$)  or is a quadratic
integer over $\rr$  \cite{BH91}.
\item
$\textnormal{tr}(\Lambda )> 0$ \cite[Appendix 4]{BH91}
\item
$\rr =\Z\textnormal{ or } \mathbb Q$ \cite{S8}.
\end{itemize}
 The general proofs in
\cite{BH91} do not give even remotely effective general bounds on the
size of a primitive matrix realizing a given nonzero spectrum.
The methods used  in \cite{S8} for the case
$\rr=\Z$ are much more tractable but still very complicated.
However, there is now an elegant  construction of Tom Laffey
\cite{Laffey2012BHTheorem}
which proves the conjecture for $\rr=\R$ in the
central special case of positive
trace, and in some other cases; where it applies,
the construction provides meaningful
bounds on the size of the realizing matrix in terms
of the spectral gap.

The nonzero spectrum of a matrix is a \lq\lq stable\rq\rq\ or
 \lq\lq eventual\rq\rq\    invariant of a matrix. For a matrix over a field,
an obvious finer invariant is the isomorphism class
of the nonnilpotent part of its action as a linear transformation.
The classification of matrices over a field by  this invariant
is the same as the  classification up to shift equivalence
 over the field;    for matrices over general rings,
from the module viewpoint (see Sec.\ref{sesect}),
shift equivalence
is the natural generalization
of the isomorphism class of this nonnilpotent linear transformation.
For some rings, an even finer invariant
 is the strong shift equivalence class. The
Generalized Spectral Conjecture  of Boyle and Handelman
(in both forms below)  heuristically is
saying that only the obvious necessary spectral
conditions constrain the eventual algebra of a primitive
matrix over a subring of $\R$, regardless of the subring
under consideration.

\begin{wgenspec}
%\cite{MR1102219}
Suppose $\rr$ is a subring of $R$ and $A$ is a square matrix over
$\rr $ whose nonzero spectrum satisfies the three necessary
conditions of the Spectral Conjecture. Then $A$ is SE over
$\rr$ to a primitive matrix.
\end{wgenspec}

\begin{sgenspec}
Suppose $\rr$ is a subring of $R$ and $A$ is a square matrix over
$\rr $ whose nonzero spectrum satisfies the three necessary
conditions of the Spectral Conjecture. Then $A$ is SSE over
$\rr$ to a primitive matrix.
\end{sgenspec}

The weak form was stated in \cite[p.253]{BH91} and
\cite[p.124]{BH93}.
The strong form was stated in
\cite[Sec. 8.4]{Boyle91matrices}), along with an
explicit admission that the authors of the conjecture did not
know if the conjectures were equivalent (not knowing if
shift equivalence over a ring implies strong shift equivalence
over it).
Following \cite{BoSc1}
(see Theorem \ref{sseclassif}), we know now that the strong form
of the Generalized Spectral Conjecture was not a vacuous generalization:
there are subrings of $\R$ over which SE does not imply
SSE (Example \ref{badrealring}).
The results of \cite{BoSc1} also provide enough structure that
we can prove Theorem \ref{twoconj}, which shows
that the two forms
of the Generalized Spectral Conjecture are equivalent.

{\bf Note!}  In  contrast to the
statement of the Generalized Spectral Conjecture for
{\it primitive} matrices,
it is {\it not} the case that
the existence of a strong shift equivalence
over $\rr$ from a matrix $A$ over $\rr$
to a nonnegative matrix can  in general be characterized
by a spectral condition on $A$.
There are dense subrings of $\R$
over which there are nilpotent matrices which are not SSE
to nonnegative matrices
(Remark \ref {nilpotentnonneg}).

There is some motivation from
symbolic dynamics for pursuing the zero trace case of the GSC.
The Kim-Roush and Wagoner primitive
matrix  counterexamples \cite{S11, Wagoner2000} to
Williams'
conjecture SE-$\Z_+$  $\implies $ SSE-$\Z_+$ rely absolutely
on certain zero-positive patterns of traces of powers of the
given matrix. We still do not know whether the refinement of
SE-$\Z_+$  by SSE-$\Z_+$
is algorithmically undecidable or (at another extreme) if it allows some finite
description involving such sign patterns. We are looking
for any related insight.

\section{Shift equivalence and strong shift equivalence} \label{sesect}
Suppose $\rr$ is a subset of a semiring and $\rr$ contains $\{0,1\}$.
(For example, $\rr$ could be  $\Z,\Z_+,\{0,1\},\R,\R_+, \ \dots \ $)
Square matrices $A,B$  over $\rr$ (not necessarily of the same
size)  are
{\it elementary strong shift equivalent over $\rr$} (ESSE-$\rr$) if there exist
matrices $U,V$ over $\rr$ such that $A=UV$ and $B=VU$.
Matrices $A,B$ are
{\it  strong shift equivalent over $\rr$} (SSE-$\rr$) if there
are a positive integer $\ell$
(the {\it lag} of the given SSE) and
 matrices $A=A_0,A_1,\dots ,A_{\ell}=B$ such that $A_{i-1}$ and $A_i$
are ESSE-$\rr$, for $1\leq i \leq \ell$. For matrices over a
subring of $\R$, the relation ESSE-$\rr$ is never transitive.
For example, if matrices
 $A,B$ are ESSE over $\R$, $j>1$ and $A^j\neq 0$, then $B^{j-1}\neq
 0$;  but if $A$ is the $n\times n$ matrix such that
$A(i,i+1)=0$ for $1\leq i <n$ and $A=0$ otherwise,
then $A$ is SSE-$\rr$ to $(0)$.
Over any ring $\rr$, the relation SSE-$\rr$ on square matrices is
generated by similarity over $\rr$ ($U^{-1}AU \sim A$) and
nilpotent extensions,
$
\left(\begin{smallmatrix} A&X\\0&0
\end{smallmatrix}\right) \sim A
\sim
\left(\begin{smallmatrix} 0&X\\0&A
\end{smallmatrix}\right) $
\cite{MallerShub1985}.

Square matrices  $A,B$ over $\rr$ are
{\it shift equivalent over $\rr$} (SE-$\rr$)
if there exist a positive integer $\ell$ and
matrices $U,V$ over $\rr$ such that the following hold:
\begin{align*}
A^{\ell} &= UV \qquad B^{\ell}=VU \\
AU &= UB \qquad BV = VA \ .
\end{align*}
Herem $\ell$ is the {\it lag} of the given SE.
It is always the case that SSE over $\rr$ implies SE over $\rr$:
from a given lag $\ell$ SSE one easily creates a lag $\ell$ SE \cite{Williams73}.
For certain semirings $\rr$, including above all $\rr=\Z_+$, the relations of
SSE and SE over $\rr$ are significant for symbolic dynamics.
The relations were introduced by Williams for the cases
 $\rr=\Z_+$ and $\rr=\{0,1\}$ to study the classification of
shifts of finite type.
Matrices over $\mathbb Z_+$ are SSE over $\Z_+$ if and only if
they define topologically conjugate shifts of finite type.
However, the relation  SSE-$\Z_+$  to this day remains
mysterious and is not even know to be decidable.
In contrast, SE-$\Z_+$
is a tractable, decidable, useful and very strong invariant of SSE-$\Z_+$.

Suppose now $\rr$ is a ring,
 and $A$ is $n\times n$ over $\rr$.
T
o see the shift equivalence relation SE-$\rr$ more conceptually,
recall that the direct limit $G_A$ of $\rr^n$ under the
$\rr$-module homomorphism  $x\mapsto Ax$
is the set of equivalence classes $[x,k]$ for $x\in \rr^n, k\in \Z_+$
under the equivalence relation $[x,k]\sim [y,j]$ if there
exists $\ell >0$ such that
$A^{j+\ell}x =A^{k+\ell}y $. $G_A$ has a well defined
group structure ($[x,k]+[y,j]=[A^kx +A^jy,j+k]$) and is an
$\rr$-module ($r: [x,k]\mapsto [xr,k]$). $A$ induces an
$\rr$-module isomorphism $\hat A: [x,k]\mapsto [Ax,k]$
with inverse $[x,k]\mapsto [x,k+1]$. $G_A$ becomes an
$\rr [t]$ module
(also an $\rr [t,t^{-1}]$ module) with $t: [x,k]\mapsto [x,k+1]$.
$A$ and $B$ are SE-$\rr$ if and only if
these $\rr [t]$-modules are
isomorphic (equivalently, if and only if they
are isomorphic as
$\rr [t,t^{-1}]$ modules).
 If the square matrix $A$ is $n\times n$, then
$I-tA$ defines a homomorphism $\rr^n \to \rr^n$ by the
usual multiplication $v\mapsto (I-tA)v$, and $\text{cok}(I-tA)$
is an $\rr[t]$-module which is isomorphic to the
$\rr[t]$-module $G_A$. For more detail and references on
these relations (by no means original to us)
 see \cite{BoSc1,LindMarcus1995}.

Williams introduced  SE and SSE in  the 1973 paper \cite{Williams73}.
For any principal ideal domain $\rr$,
Effros showed SE-$\rr$ implies SSE-$\rr$
in the 1981 monograph \cite{Effros1981}
(see \cite{Williams1992}
for Williams' proof for the case $\rr = \Z$).
In the 1993 paper \cite{BH93},
Boyle and Handelman extended this result to the case that
$\rr$ is a Dedekind domain (or, a little  more generally,  a Pr{\"u}fer domain).
Otherwise, the relationship of SE and SSE of matrices
over a ring remained open until the recent
paper \cite{BoSc1}, which
explains the relationship in general as follows.

\begin{theorem} \label{sseclassif} \cite{BoSc1}
Suppose $A,B$ are SE over a ring $\rr$.
\begin{enumerate}
\item
There is a nilpotent matrix $N$ over $\rr$
such that  $B$ is SSE over $\rr$ to the
matrix
$A\oplus N=
\begin{pmatrix} A & 0 \\ 0 & N
\end{pmatrix} $.
\item
The map $[I-tN] \to [A \oplus N]_{SSE}$ induces a bijection from
$\textnormal{NK}_{1}(\mathcal R)$ to the set of SSE classes of matrices
over $\rr$ which are in the SE-$\rr$ class of $A$.
\end{enumerate}
\end{theorem}
We will say just a little now about
$\textnormal{NK}_1(\rr)$, a  group of great importance in algebraic
$K$-theory;
for more background, we have found \cite{Ranickibook,Rosenberg1994, WeibelBook}
very helpful.
$\textnormal{NK}_1(\rr)$ is the kernel of the map $K_1(\rr [t]) \to K_1(\rr )$
induced by the ring homomorphism $\rr [t] \to \rr $ which
sends $t$ to $0$. The finite matrix $I-tN$ corresponds to
the matrix $I-(tN)_{\infty}$ in the group $\textnormal{GL}(\rr [t])$
(with $I$ denoting the $\mathbb N \times \mathbb N$  identity matrix and
$(tN)_{\infty}$ the
$\mathbb N \times \mathbb N$ matrix which agrees with $tN$ in an upper left
corner and is otherwise zero). Every class of $\textnormal{NK}_1(\rr
)$ contains a   matrix of the form
$I-(tN)_{\infty}$ with $N$ nilpotent over $\rr$.
$\textnormal{NK}_1(\rr)$ is trivial for many rings
(e.g., any field, or more generally any left regular Noetherian
ring) but not for all rings. If
$\textnormal{NK}_1(\rr)$ is not trivial, then it is not finitely generated
as a group. From the established theory, it is easy
to give an example of
a subring $\rr$ of $\R$ for which $\textnormal{NK}_1(\rr)$ is not trivial
(Example \ref{badrealring}).

\section{Proof of  Theorem \ref{twoconj}}
\begin{proof}[Proof of Theorem \ref{twoconj}]
Given a square matrix $M$ over $\R$, let $\lambda_M$ denote its
spectral radius and define the matrix $|M|$ by $|M|(i,j)= |M(i,j)|$.

By Theorem \ref{sseclassif},
let $N$ be a nilpotent  matrix
such that $B$  is SSE over $\rr$ to
the matrix
$\begin{pmatrix} A&0\\0&N
\end{pmatrix} $ .
Suppose $M$ is a matrix
SSE over $\rr$ to $N$ and $M$ also satisfies the following
conditions:
\begin{enumerate} \label{conditions}
\item
$\lambda_{|3M|}<\lambda_A $
\item
For all positive integers $n$,
$\text{trace}(|3M|^n)\leq
\text{trace}(A^n)
$ .
\item
For all positive integers $n$ and $k$,
if $\text{tr}(|3M|^n)<\text{tr}(A^n)$,
then
$\text{tr}(|3M|^{nk})<\text{tr}(A^{nk})$.
\end{enumerate}
Then by the Submatrix Theorem (Theorem 3.1 of \cite{BH91}),
 there is a primitive
matrix $C$ SSE over $\rr$ to $A$ such that
$|3M|$ is a proper principal submatrix of $C$.
Without loss of generality, let this submatrix
occupy the upper left corner of $C$.
Define $M_0$ to be the matrix of size matching
$C$ which is $M$ in its upper left corner and which
is zero in other entries. Then $B$ is SSE over $\rr$
to the matrix
$\begin{pmatrix} C&0\\0&M_0
\end{pmatrix}$. Choose $\epsilon \in \rr$ such that
$1/3 < \epsilon < 2/3$ and compute
\begin{align*}
\begin{pmatrix}  I &-\epsilon I\\0&I\end{pmatrix}
\begin{pmatrix}  C &0\\0&M_0\end{pmatrix}
\begin{pmatrix}  I &\epsilon I\\0&I\end{pmatrix}
&=
\begin{pmatrix}  C &\epsilon ( C - M_0 )\\0&M_0\end{pmatrix}
  \\
\begin{pmatrix}  I &0\\ I&I\end{pmatrix}
\begin{pmatrix}  C &\epsilon (C - M_0 )\\0&M_0\end{pmatrix}
\begin{pmatrix}  I &0\\- I&I\end{pmatrix}
&=
\begin{pmatrix}  (1-\epsilon )C +\epsilon M_0 &\epsilon
(C - M_0)
\\(1-\epsilon )(C-M_0 )&\
\epsilon C +(1-\epsilon )M_0 \end{pmatrix} := G\ .
\end{align*}
The matrix $G$ is SSE over $\rr$ to $B$, and it is nonnegative.
The diagonal blocks have positive entries wherever $C$ does;
 because $C$
is primitive, there is a $j>0$ such that $C^j>0$, and
therefore the
diagonal blocks of $G^j$ are also positive.
Because neither offdiagonal block of $G$  is the zero block,
it follows that  $G$ is primitive.

So, it suffices to find $M$ SSE over $\rr$ to $N$ satisfying the
conditions
%\eqref{conditions}.
(1)-(3) above.
Choose $K$ such that $\text{tr}(A^k)>0 $ for all $k> K$.
Let $n$ be the integer such that
$N$ is $n\times n$, and let $J$
by the integer provided by  Proposition \ref{nilpotentproposition}
given $n$ and $K$. Given this $J$,
choose $\epsilon >0$ such that for any
$J\times  J$ matrix $M$ with $||M||_{\infty}< \epsilon$,
we have $\lambda_{3|M|}< \lambda_A$ and for
$k>K$ we also have $\text{tr}(|3M|^k)< \text{tr}(A^k)$.
Now let $\delta>0$ be as provided by Proposition \ref{nilpotentproposition}
for this $\epsilon $.

If we can now find an $n\times n$ nilpotent matrix $N'$ which is
SSE over $\rr$ to $N$ and satisfies $||N'||<\delta$, then
we can apply Proposition \ref{nilpotentproposition} to this
$N'$ to produce a matrix $M$ SSE over $\rr$ to $N$ and
with $||M||<\epsilon $ and with $\text{tr}(M^k)=0$ for $1\leq k \leq K$.
This matrix $M$ will satisfy the conditions (1)-(3).

Pick $\gamma >0$ such that
$||\gamma N||_{\infty} < \delta$.
There is a matrix $U$ in $\text{SL}(n,\R )$ such that
$U^{-1}NU = \gamma N$. The matrix $U$ is a product of basic elementary
matrices over $\R$, and these can be approximated arbitrarily
closely by basic elementary matrices over $\rr$. Consequently there
is a matrix $V$ in $\text{SL}(n,\rr )$ such that
$||V^{-1}NV||_{\infty} < \delta $. Choose
$N' = V^{-1}NV$.
\end{proof}

To prove the Proposition
\ref{nilpotentproposition} on which the proof of Theorem \ref{twoconj}
depends, we use a correspondence proved in \cite{BoSc1}.
We need some definitions.

Given a finite matrix $A$, let $A_{\infty}$ denote the
$\N \times \N$ matrix which has $A$ as its upper left
corner and is otherwise zero. In any $\N \times \N$ matrix,
$I$ denotes the infinite identity matrix.
% To avoid a heavier notation, we use the same
%letter to refer to a finite matrix and its infinite version.
%For example, give a finite matrix $A$, we let
%$I-A$ denote the matrix $I-A_{\infty}$; and an $\N \times \N$
%matrix $I-A$ always means a matrix $I-A_{\infty}$.
Given a ring $R$, $\text{El}(R)$ is the group of
$\N\times \N$ matrices over $R[t]$, equal to the infinite
identity matrix except in finitely many entries, which are
products of basic elementary matrices (these basic
matrices are by definition equal to $I$ except
perhaps in a single offdiagonal entry).
For finite matrices $A,B$,
the matrices $I-A_{\infty}$ and $I-B_{\infty}$ are
$\text{El}(R[t])$ equivalent if there  are matrices
$U,V$ in $\text{El}(R[t])$ such that
$U(I-A_{\infty})V=I-B_{\infty}$.

\begin{definition} \label{asharpdefinition}
Given a finite matrix $A$ over $t\rr [t]$, choose $n\in \N$ and $k\in \N$
such that $A_1, \dots A_k$ are $n\times n$ matrices over
$\rr$ such that
\[
A_{\infty}  = \sum_{i=1}^k t^i(A_i)_{\infty}
\]
and define a finite matrix $\As= \mathcal A^{\sharp (k,n)}$ over $\rr$ by the
following block form,
in which  every block is $n\times n$:
\[
 \As =
\begin{pmatrix}
A_1 & A_2 &A_3& \dots &A_{k-2}&A_{k-1} & A_k \\
I   & 0   &0  & \dots & 0     & 0      & 0  \\
0   & I   &0  & \dots & 0     & 0      & 0  \\
0   & 0   & I & \dots & 0     & 0      & 0  \\
\dots &\dots &\dots &\dots &\dots &\dots &\dots  \\
0   & 0   & 0 & \dots &I      & 0     & 0  \\
0   & 0   & 0 & \dots &0      & I     & 0
\end{pmatrix} \ .
\]
\end{definition}
In the definition, there is some freedom in
the choice of $\As$: $k$ can be increased by
using zero matrices, and $n$ can be increased
by filling additional entries of the $A_i$
with zero. These choices do not affect the
SSE-$\rr$ class of $\As$.

\begin{theorem}\label{finecentral}\cite{BoSc1}
 Let $\rr$ be a ring. Then there is a
 bijection between the following sets:
\begin{itemize}
\item the set of
$\textnormal{El}(\rr [t])$ equivalence classes of
$\N \times \N$ matrices $I-A_{\infty}$ such that  $A$
is a finite matrix over $t\rr [t]$
% \mbred{Changed $t\mathbb{Z}[t]$ to $t \rr[t]$}
\item
the set of SSE-$\rr$ classes of square matrices
 over $\rr$.
\end{itemize}
The bijection from $\textnormal{El}(\rr [t])$ equivalence classes
  to SSE-$\rr$ classes is induced by
the map
 $ I-A_{\infty}  \mapsto \As$. The inverse map (from
the set of SSE-$\rr$ classes)
 is induced by the map sending $A$ over $\rr$ to
the $\N \times \N $ matrix
$ I-tA$.
\end{theorem}

By the degree of a  matrix with polynomial entries we mean the maximum
degree of
its entries. If $M$ is a matrix over $\R [t]$,
with entries $M(i,j)= \sum_{i,j,k} m_{ijk}t^k$,
then we define
$||M||= \max_{k>0} \max_{i,j} |m_{ijk}|$. If $M$ is a matrix over $\R$,
with $M(i,j)=  m_{ij}$,
then $||M||_{\infty} $ is the usual sup norm,
$||M||_{\infty}= \max_{i,j} |m_{ij}|$.

%
%\mbred{I'm a little confused by the notation below in 3.3. It looks
%  like the entries should be $a_{ij} = \sum_{0\leq r \leq d}
%  a_{ij}^{(r)} t^{k+r} $, instead of just the $t^{r}$, since $A$ is
%  over $t^{k}\Rcal[t]$? Also, I can't see where the $\j \in
%  \mathbb{N}$ comes in now.
%MB: I commented out the $\j$ part of the theorem statement --
%I think it was just something I failed to clean out earlier. For the notation,
%as it stands  $a_{ij}^{(r)}=0$ if $r<k$, but I don't see an
%inconsistency.}
\begin{lemma}\label{nilpotentlemma}
Suppose $\rr$ is a dense subring of $\R$
% $\j \in \mathbb N$
 and $A$ is an $n\times n$   matrix
of degree $d$
over $t^k\rr [t]$, with entries $a_{ij} = \sum_{1\leq r \leq d} a_{ij}^{(r)} t^r $.
Suppose $\sum_{i=1}^na^{(k)}_{ii}=0$ and
$||A|| \leq \frac 1{4n^2}$.
Then there is
an $n\times n$
matrix $B$ over $t^{k+1}\rr [t]$ such that
$I-A_{\infty}$ is
$\textnormal{El}(\rr [t])$ equivalent to
$I-B_{\infty}$ and
the following hold:
\begin{enumerate}
\item
$\textnormal{degree}(B) \ \leq \
\textnormal{degree}(A) + 3k
$.
\item $||B||\  \leq \ 4n^3||A||$ .
\end{enumerate}
\end{lemma}

\begin{proof}
For finite square matrices $I-C$ and  $I-D$, we
use $I-C \sim I-D$ to denote elementary equivalence
over $\rr [t]$ of
$I-C_{\infty}$ and
$I-D_{\infty}$.  We have
\begin{align*}
I-A \ &=\
\begin{pmatrix}
1-a_{11} & -a_{12} & \cdots & -a_{1n} \\
-a_{21} & 1-a_{22} & \cdots & -a_{2n} \\
\vdots & & \ddots & \vdots \\
-a_{n1} & -a_{n2} & \cdots & 1-a_{nn} \end{pmatrix} \\
&\sim  \
\begin{pmatrix} 1-a_{11} & -a_{12} & \cdots & -a_{1n} & a^{(k)}_{11}t^k \\
-a_{21} & 1-a_{22} & \cdots & -a_{2n}& a^{(k)}_{22}t^k \\
\vdots & & \ddots & \vdots & \vdots \\
-a_{n1} & -a_{n2} & \cdots & 1-a_{nn} &  a^{(k)}_{nn}t^k  \\
0 & 0 & \cdots & 0 & 1\end{pmatrix} := I-A_1 \ .
\end{align*}
In order, apply the following elementary operations:
\begin{enumerate}
\item
For $1\leq j\leq n$, add column $n+1$ to column $j$ of $I-A_1$, to produce
a matrix $I-A_2$.
Then $\textnormal{degree} (A_2)= \textnormal{degree} (A)$;
the diagonal entries of $A_2$ lie in $t^{k+1}\rr [t]$;
 and
 $||A_2||\leq 2||A_1||=2||A||$.
Every entry in row $n+1$ of $I-A_2$ equals 1.
(By definition these entries have no impact on $||A_2||$.)
\item
For $1\leq i \leq n$, add (-1)(row $i$) of $(I-A_2)$
to row $n+1$ to form
$I-A_3$. Then the entries of
 $A_3$ lie in $t^k \rr [t]$, and the diagonal entries of
$A_3$ still lie in  $t^{k+1} \rr [t]$, since
 $\sum_{i=1}^na^{(k)}_{ii}=0$ .
We have  $||A_3||\leq n||A_2||\leq 2n||A||<1$
 and
 $\textnormal{degree} (A_3)\leq \textnormal{degree} (A)$ .
\item
For $1\leq i \leq n$, add ($-a^{(k)}_{ii}t^k$)(row $n+1$) of $(I-A_3)$
to row $i$ to form $I-A_4$.
In block form,
\[
I-A_4 \ = \
\begin{pmatrix} I-A_5 & 0 \\ x & 1
\end{pmatrix}
\]
in which $A_5$ is $n\times n$ and
$x=(x_1\cdots x_n)$. Adding multiples of column $n+1$ to
columns $1, \dots , n$ to clear out $x$, we see
$I-A_5\sim I-A$. We have
$\textnormal{degree} (A_5)\leq \textnormal{degree} (A)+k $ and
\begin{align*}
||A_5||\ &
\leq \ || A_3|| + (|| A||)(  ||A_3|| ) \\
&\leq \ 2||A_3||\  \leq \ 4n ||A||\ < \ 1 \ .
\end{align*}
In $A_5$, the diagonal terms  lie in  $t^{k+1}\rr [t]$ and
the offdiagonal terms
lie in $t^{k}\rr [t]$.
In the next two steps,
we apply elementary
 operations to clear the degree $k$ terms outside the diagonal.
We use part of a  clearing algorithm from \cite{S12}.

\item
Let $b_{ij}$ be the coefficient of $t^k$ in $A_5(i,j)$.
For $2\leq i \leq n$, add $(-b_{1j}t^k)(\text{row }j)$ to
row 1. Continuing in order for rows $i=2, \dots , n-1$:
for $i+1\leq j \leq n$, add
$(-b_{ij}t^k)(\text{row }j)$ to
row $i$. Let $(I-A_6)$ be the resulting matrix.
The entries of $A_6$ on and above the diagonal lie in
$t^{k+1}\rr [t]$. We have
\[
\textnormal{degree} (A_6)\ \leq \ \textnormal{degree} (A_5)+k
\ \leq \ \textnormal{degree} (A)+2k
\]
and
\begin{align*}
||A_6||\ & \leq \ ||A_5|| + (n-1)||A_5||^2 \\
& \leq \ n||A_5||
\ \leq \ 4n^2 ||A|| \leq 1 \ .
\end{align*}
\item
Let $c_{ij}$ denote the coefficient of $t^k$ in
$A_6(i,j)$. For  $2\leq j \leq n$, add $(-c_{j1}t^k)(\text{column }j)$
of $A_6$ to column 1. Continuing in order for columns $i=2, \dots , n-1$:
for $i+1\leq j \leq n$, add
$(-c_{ji})(\text{column }j)$ to
column $i$. For the resulting matrix $(I-B)$,
the entries of $B$ lie in $t^{k+1}\rr [t]$, with
\[
\textnormal{degree} (B)
\ \leq \ \textnormal{degree} (A_6)+k \
\leq \ \textnormal{degree} (A)+3k
\]
 and
\begin{align*}
||B||\ & \ \leq \ ||A_6||+ (n-1) ||A_6||^2 \\
&\ \leq \ n ||A_6||\  \leq \ 4n^3
||A|| \ .
\end{align*}
\end{enumerate}
\end{proof}

\begin{proposition} \label{nilpotentproposition}
Suppose $\rr$ is a dense subring of $\R$,
  $n\in \N$ and $K\in \N$.
Then there is a $J$ in $\N$ such that for any $\epsilon >0$
there exists $\delta>0$ such that the following holds: \\
if $N$ is a  nilpotent $n\times n$ matrix over $\rr$ and
$||N||_{\infty}<\delta$, then there is a $J\times J$ matrix $M$
over $\rr$ such that
\begin{enumerate}
\item
$M$ is SSE over $\rr$ to $N$,
\item
$\textnormal{tr}\, (|M|^k) =0$ for $1\leq k \leq K$, and
\item
$||M||_{\infty}< \epsilon $ .
\end{enumerate}
\end{proposition}

\begin{proof}
Because $N$ is nilpotent, $\text{tr}(N^k)=0$ for all
positive integers $k$.
Set $B_0=tN$. We define matrices $B_1, \dots ,B_K$ recursively,
letting $I-B_{k+1}$ be the matrix $I-B$ provided by
Lemma \ref{nilpotentlemma} from input $I-A=I-B_{k}$.
The conditions of the lemma are satisfied recursively, because
the (zero) trace of the $k$th power of the nilpotent matrix $(B_k)^{\sharp}$
must be (in the terminology of the lemma) $\sum_i a_{ii}^{(k)}$.
The matrix $B_{K}$ is $n\times n$
with  entries of degree at most
\[
d:= 1 +3(1) + 3(2) + \dots + 3(K)
= 1 + 3K(K+1)/2 \ .
\]
Let
$(B_K)_i$ be the matrices, $1\leq i \leq d$, such that
$B_K=
\sum_{i=1}^d (B_K)_i t^i$ .
Define $M$ to be the matrix
$(B_K)^{\sharp}$, an  $nd\times nd$ matrix over $\rr$
which is SSE over $\rr$ to $N$. Set $J=
nd $.

It is now clear from condition $(2)$
of Lemma \ref{nilpotentlemma} and induction
that given $\epsilon > 0$, there is a $\delta >0$
such that $||N||<\delta $ implies
$|| (B_K)||< \epsilon$. (We are not trying to
optimize estimates.) With $K>1$ (without loss of
generality), we have $|| B_K||= || (B_K)^{\sharp}||_{\infty}$.
This finishes the proof.
\end{proof}

\begin{example} \label{badrealring}
%\red{In the first paragraph of the example, everywhere I changed $s$
%  to $z$, to match the choice of variables for the displayed matrix.}
% {\normalfont
 {\normalfont
There are subrings of $\R$ with nontrivial $\textnormal{NK}_1$.
For example, let $\rr=\mathbb Q [t^2,t^3,z,z^{-1}]$. By the Bass-Heller-Swan Theorem (see \cite{Rosenberg1994}, 3.2.22) for any ring $\mathcal{S}$, there is a splitting $K_{1}(\mathcal{S}[z,z^{-1}]) \cong K_{1}(\mathcal{S}) \oplus K_{0}(\mathcal{S}) \oplus \textnormal{NK}_{1}(\mathcal{S}) \oplus \textnormal{NK}_{1}(\mathcal{S})$, which implies $\textnormal{NK}_{1}(\mathcal{S}[z,z^{-1}])$ always contains a copy of $\textnormal{NK}_{0}(\mathcal{S})$. An elementary argument (see for example exercise 3.2.24 in \cite{Rosenberg1994}) shows that $\textnormal{NK}_{0}(\mathbb{Q}[t^{2},t^{3}]) \ne 0$, so $\textnormal{NK}_{1}(\mathbb{Q}[t^{2},t^{3},z,z^{-1}])$ is non-zero. Since
$\mathbb{Q}[t^{2},t^{3},z,z^{-1}]$ can be realized as a subring of $\mathbb{R}$
(by an embedding
 sending $t,z$ to
algebraically independent
transcendentals in $\mathbb{R}$)
this provides an example of a subring
$\rr$  of $\mathbb{R}$ for which $\textnormal{NK}_{1}(\Rcal)$ is not zero, and
therefore shift equivalence over $\rr$
does not imply strong shift equivalence
 over $\rr$. \\

It is possible
% (although complicated)
to produce explicit
 examples
%of matrices representing nontrivial elements in
by tracking through the
exact sequences behind the argument of the last paragraph.
This is done in \cite{SchmiedingExamples},
and for $\rr = \mathbb{Q}[t^{2},t^{3},z,z^{-1}]$
yields the following matrix
over $\rr [s]$,
\[
I-M=\begin{pmatrix} 1-(1-z^{-1})s^{4}t^{4} & (z-1)(s^{2}t^{2}-s^{3}t^{3})
  \\ (1-z^{-1})(s^{2}t^{2})(1+st+s^{2}t^{2}+s^{3}t^{3}) &
  1+(z-1)(s^{4}t^{4}) \end{pmatrix} \ ,
\]
 which is nontrivial as an element
of $\textnormal{NK}_1(\rr )$.
Writing $M$ as
\[
M=\begin{pmatrix}
(1-z^{-1})s^{4}t^{4} & (1-z)(s^{2}t^{2}-s^{3}t^{3})  \\
(z^{-1}-1)(s^{2}t^{2})(1+st+s^{2}t^{2}+s^{3}t^{3}) &
(1-z)(s^{4}t^{4}) \end{pmatrix}
= \sum_{i=1}^5 s^iM_i
\]
with the
$M_i$ over $\rr$, we obtain (see \cite{BoSc1}) a nilpotent matrix $N$ over $\rr$,
\begin{align*}
&N\ =\
\begin{pmatrix}
M_1 & M_2 & M_3 & M_4 & M_5 \\
I       &  0    &   0    &  0    &  0    \\
0      &  I     &   0    &  0    &  0    \\
0      &  0    &   I    &  0    &  0    \\
0      &  0    &   0    &  I    &  0
\end{pmatrix} \ =\ \\
&
\begin{pmatrix}
0&0 &         0&(1-z)t^2 &     0&(1-z)(-t^3) & (1-z^{-1})t^4&0 & 0&0\\
0&0 & (z^{-1}-1)t^2&0 & (z^{-1}-1)t^3&0 & (z^{-1}-1)t^4&(1-z)t^4 & (z^{-1}-1)t^5&0\\
1&0 & 0&0 & 0&0 & 0&0 &0&0 \\
0&1 & 0&0 & 0&0 & 0&0 &0&0 \\
0&0 & 1&0 & 0&0 & 0&0 &0&0 \\
0&0 & 0&1 & 0&0 & 0&0 &0&0 \\
0&0 & 0&0 & 1&0 & 0&0 &0&0 \\
0&0 & 0&0 & 0&1 & 0&0 &0&0 \\
0&0 & 0&0 & 0&0 & 1&0 &0&0 \\
0&0 & 0&0 & 0&0 & 0&1 &0&0
\end{pmatrix}
\end{align*}
which is nontrivial as an element of $\textnormal{Nil}_{0}(\rr)$, as
is the matrix
$N'$  obtained
by
removing the last row and the last column from $N$.
%(The awkwardness
%of the big matrix $N$ indicates appeal of using $M$ as a polynomial
%presentation of $N$.)

%BELOW ARE SOME OF THE MATRIX COMPUTATIONS ALONG THE WAY
%TO WRITING OUT N.
%\[
%I-M=\begin{pmatrix} 1-(1-z^{-1})s^{4}t^{4} & (z-1)(s^{2}t^{2}-s^{3}t^{3})
%  \\ (1-z^{-1})(s^{2}t^{2})(1+st+s^{2}t^{2}+s^{3}t^{3}) &
%  1+(z-1)(s^{4}t^{4}) \end{pmatrix} \ ,
%\]
%
%\[
%M=\begin{pmatrix}
%(1-z^{-1})s^{4}t^{4} & (1-z)(s^{2}t^{2}-s^{3}t^{3})  \\
%(z^{-1}-1)(s^{2}t^{2})(1+st+s^{2}t^{2}+s^{3}t^{3}) &
%(1-z)(s^{4}t^{4}) \end{pmatrix} \ ,
%\]
%
%$M_1=0$
%
%\[
%s^2M_2=
%s^2\begin{pmatrix}
%0 & (1-z)t^{2}  \\
%(z^{-1}-1)t^{2} &
%0\end{pmatrix} \ ,
%\]
%
%\[
%s^3M_3=
%s^3\begin{pmatrix}
%0& (1-z)(-t^{3})  \\
%(z^{-1}-1)t^{3}&
%0\end{pmatrix} \ ,
%\]
%
%\[
%s^4M_4=
%s^4\begin{pmatrix}
%(1-z^{-1})t^{4} & 0 \\
%(z^{-1}-1)t^{4} &
%(1-z)t^{4} \end{pmatrix} \ ,
%\]
%
%\[
%s^5M_5=
%s^5\begin{pmatrix}
%0& 0 \\
%(z^{-1}-1)t^{5} &
%0\end{pmatrix} \ ,
%\]
%
The matrix $N'$ is $9\times 9$. We don't have a smaller example, and
we don't have a decent example of two positive matrices which are
shift equivalent but not strong shift equivalent over a subring of
$\mathbb R$.
}
%}
\end{example}

\begin{remark} \label{nilpotentnonneg}
 {\normalfont
Suppose $\rr$ is a subring of $\R$ and $N$ is a nonnegative
nilpotent matrix over $\rr$. Then there is a permutation
matrix $P$ such that $P^{-1}NP$ is triangular with zero diagonal.
Using elementary SSEs of the block form
\[
\begin{pmatrix} X&Y\\0&0
\end{pmatrix}
=
\begin{pmatrix} I\\0
\end{pmatrix}
\begin{pmatrix} X&Y
\end{pmatrix}
\qquad \text{and} \qquad
\begin{pmatrix} X
\end{pmatrix}
=
\begin{pmatrix} X&Y
\end{pmatrix}
\begin{pmatrix} I\\0
\end{pmatrix}
\]
we see that  $P^{-1}NP$ (and hence $N$)
is SSE over $\rr$ to $[0]$.
By Theorem \ref{sseclassif},  with $A=0$,  it follows that
a nilpotent matrix $N$ is
SSE over $\rr$ to a nonnegative matrix if and only if $[I-tN_{\infty} ]$
is trivial in $\textnormal{NK}_1(\rr  )$.
Therefore,
if (and only if)
$\textnormal{NK}_1(\rr )$ is nontrivial, there will be
nilpotent matrices over $\rr$ which cannot be SSE over $\rr$ to a
nonnegative matrix.
The matrix $N$ in Example \ref{badrealring} is one such example.
}
\end{remark}

\section{Reflections on the Generalized Spectral Conjecture}

%\mbred{
%re subtlety of perturbing perron value \cite{CroninLaffey2012,LLS2009} }

Is the Generalized Spectral Conjecture true?

For $\rr=\Z$, the Spectral Conjecture is true \cite{S8}.
The GSC is true for $\rr=\Z$ for a given
$\Delta$ if every entry of $\Delta$ is a rational
integer \cite{BH93}. There is not much more
direct evidence for the GSC for $\rr=\Z$,
but we know of no results which
cast doubt.

From here,  suppose $\rr$ is a dense subring of $\R$.
As noted earlier, the Spectral Conjecture is almost surely true.
Theorem \ref{twoconj} removes the possibility  that
 the very subtle algebraic invariants
following from Theorem
\ref{sseclassif} could be an obstruction to the GSC.
The GSC was proved in \cite{BH93} in the following cases:
\begin{enumerate}
\item
when the nonzero spectrum is contained in $\rr$, and
$\rr$ is a Dedekind domain with a nontrivial unit;
\item \label{postrpro}
when the nonzero spectrum has positive trace and
either (i) the spectrum is real or (ii)
the minimal and characteristic polynomials
of the given matrix are equal up to a power of the
indeterminate.
\end{enumerate}

The following Proposition (almost  explicit in
\cite[Appendix 4]{BH91}) is more evidence for the
GSC  in the positive trace case.

\begin{proposition} \label{realreduction}
Suppose the Generalized Spectral Conjecture holds for
matrices of positive trace for the ring $\R$. Then it
holds for matrices of positive trace for every dense subring
$\rr$ of $\R$.
\end{proposition}
\begin{proof}
Let $A$ be a square matrix over $\rr$ of positive trace
which over $\R$
is SSE to a primitive real matrix $B$.
We need to show that $A$ is SSE over $\rr$ to a primitive
matrix.

By \cite{S32} (or the alternate exposition
\cite[Appendix B]{BKR2013}), because $B$ is primitive with positive
trace, there is a positive matrix $B_1$ SSE over $\R$
(in fact over $\R_+$) to
$B$. And then, by arguments in \cite{S32},
for some $m$ there are $m\times m$ matrices
$A_2,B_2$ (obtained through
row splittings of $A$ and $B_1$ ),
with $B_2$ positive,  such that $A$ is SSE over $\rr$ to $A_2$;
$B_1$ is SSE over $\R$
(in fact over $\R_+$)
to a positive
matrix $B_2$; and there is a matrix $U$
in $\textnormal{SL}(m, \R)$ such that $U^{-1}A_2U =B_2$.
Because $\textnormal{SL}(m, \rr)$ is dense in $\textnormal{SL}(m, \R)$,
%Because $\rr$ is dense in $\R$ and elements of
%$\textnormal{SL}(m, \R)$ are products of basic elementary
%matrices,
%$A'$ and $B'$ are conjugate
%(alternately see \cite{S0})
and $B_2$ is positive, there is a $V$ in
$\textnormal{SL}(m, \rr)$ such that $V^{-1}A_2V$ is positive.
This matrix $(V^{-1}A_2)(V)$ is SSE over $\rr$ to
the matrix $(V)(V^{-1}A_2)=A$.
\end{proof}

After more than 20 years, the GSC
remains open even in the case $\rr = \mathbb R$.
Still, the GSC seems correct.
What we lack is a proof.

\bibliographystyle{plain}
\bibliography{mbssbib}

\appendix

\section{Correction}

The preceding version  of ``Strong shift equivalence and the
generalized spectral conjecture for nonnegative matrices''
is essentially the same as the arxiv version 1. 
The purpose of this post is  to communicate a correction, which
we state separately since the paper has been published 
(DOI 10.1016/j.laa.2015.06.004),  in Linear Algebra and its Applications.  

Theorem \ref{sseclassif} states   a  result claimed in
the version 1 post of 
\cite{BoSc1}.  This quoted result was corrected in
the version 2  post  of 
\cite{BoSc1} (to appear in Crelle's Journal): the bijection
of Theorem 2.1(2) in general
is not to $\text{NK}_1(\calR )$, but to a certain quotient
group $\text{NK}_1(\calR )/\ear$.
The ``elementary stabilizer'' $\ear$ need not be trivial,
but is trivial in many cases (for example, if $A$ is invertible over
$\calR$). 

With the corrected reference,
the proof of equivalence of the strong and weak forms
of the Spectral Conjecture (i.e., Theorem \ref{twoconj}) goes
through without change.
Also, because $\ear$ is in many cases trivial,
it still holds
(as discussed  after the statements of the weak and
strong forms of the conjecture, bottom of page 3)
that  SE-$\calR$ in general doesn't
imply SSE-$\calR$, and therefore the  strong form of the conjecture was
not vacuously equivalent to the weak form.

\end{document}